\documentclass[reqno]{amsart}%
\usepackage{amsmath}
\usepackage{amsfonts}
\usepackage{CJK}
\usepackage{amssymb}
\usepackage{graphicx}%
\usepackage{hyperref}
\usepackage{amscd}
\usepackage{mathrsfs}
\usepackage{bbm}
\usepackage{color}

\newtheorem{thm}{Theorem}[section]
\newtheorem{cor}{Corollary}[section]
\newtheorem{lem}[thm]{Lemma}

\theoremstyle{definition}

\theoremstyle{Conjecture}

\theoremstyle{remark}
\newtheorem{rem}{Remark}[section]
\theoremstyle{Example}

\newcommand{\be}{\begin{equation}}
\newcommand{\ee}{\end{equation}}
\newcommand{\bea}{\begin{eqnarray}}
\newcommand{\eea}{\end{eqnarray}}
\newcommand{\ben}{\begin{eqnarray*}}
\newcommand{\een}{\end{eqnarray*}}
\newcommand{\bet}{\begin{equation}
\begin{split}}
\newcommand{\eet}{\end{split}
\end{equation}}

\begin{document}
\title[Integration estimation related to Strong Openness Conjecture]
{Estimation of weighted $L^2$ norm related to Demailly's Strong Openness Conjecture}

\author{Qi'an Guan}
\address{Qi'an Guan: School of Mathematical Sciences, and Beijing International Center for Mathematical Research,
Peking University, Beijing, 100871, China.}
\email{guanqian@amss.ac.cn}
\author{Zhenqian Li}
\address{Zhenqian Li: School of Mathematical Sciences, Peking University, Beijing, 100871, China.}
\email{lizhenqian@amss.ac.cn}
\author{Xiangyu Zhou}
\address{Xiangyu Zhou: Institute of Mathematics, AMSS, and Hua Loo-Keng Key Laboratory of Mathematics, Chinese Academy of Sciences, Beijing, 100190, China}
\email{xyzhou@math.ac.cn}

\thanks{The authors were partially supported by NSFC-11431013. The third author would like to thank NTNU for offering him Onsager Professorship.
The first author was partially supported by NSFC-11522101.}

\date{\today}
\subjclass[2010]{32C35, 32L10, 32U05, 32W05}
\thanks{\emph{Key words}. Strong openness conjecture, Effectiveness, Plurisubharmonic function, Multiplier ideal
       sheaf}

\begin{abstract}
In the present article, we obtain an estimation of the weighted $L^2$ norm near the singularities of plurisubharmonic weight related to Demailly's strong openness conjecture, which implies the convergence of the weighted $L^2$ norm.
\end{abstract}
\maketitle

\section{introduction}\label{sec:introduction}
Let $D\subset\mathbb{C}^n$ be a bounded pseudoconvex domain, $o\in D$ the origin of $\mathbb{C}^n$ and $\varphi\in Psh(D)$ a plurisubharmonic function on $D$. The multiplier ideal sheaf $\mathscr{I}(\varphi)$ consists of germs of holomorphic functions $f$ such that $|f|^2e^{-\varphi}$ is locally integrable, which is a coherent sheaf of ideals (see \cite{De}).

\textbf{Demailly's strong openness conjecture (SOC) \cite{De_ICTP}:} If $(f,o)\in\mathscr{I}(\varphi)_o$, then there exists $\varepsilon>0$ such that $(f,o)\in\mathscr{I}((1+\varepsilon)\varphi)_o$.

Note that $\mathscr{I}(\varphi)_o$ is finitely generated by $(f_j)_{j=1,...,k_0}$. Let $\mathscr{I}(\varphi)_o=(f_1,...,f_{k_0})$. The truth of SOC implies that there exists $\varepsilon_j>0$ such that $(f_j,o)\in\mathscr{I}((1+\varepsilon_j)\varphi)_o$ for any $1\leq j\leq k_0$. Then, SOC is equivalent to $\mathscr{I}(\varphi)_o=\mathscr{I}((1+\varepsilon_0)\varphi)_o$, where $\varepsilon_0=\min\limits_{1\leq j\leq k_0}\{\varepsilon_j\}$.
As $\mathscr{I}((1+\varepsilon_0)\varphi)_o\subset\mathscr{I}_+(\varphi)_o:=\cup_{\varepsilon>0}
\mathscr{I}((1+\varepsilon)\varphi)_o\subset\mathscr{I}(\varphi)_o$, then SOC is equivalent to
$\mathscr{I}(\varphi)_o=\mathscr{I}_+(\varphi)_o$.

In \cite{G-Z_open}, Guan and Zhou proved the above SOC. Moreover, they also established an effectiveness about $\varepsilon$ of the conjecture in \cite{G-Z_effect}.

Let $L^2_{\mathcal{O}}(D)$ be the Hilbert space of homomorphic functions on $D$ with finite $L^2$ norm
$$L^2_{\mathcal{O}}(D):=\{f\in\mathcal{O}(D)\big|\ ||f||^2_D=\int_D|f|^2d\lambda_n<\infty\}.$$
whose inner product is defined to be $(f,g)=\int_Df\cdot\overline gd\lambda_n$.

Let $I\subset\mathcal{O}_o$ be an ideal and $(e_k)_{k\in\mathbb{N}^+}$ an orthonormal basis of $$\mathcal{H}_I:=\{f{\in}L^2_{\mathcal{O}}(D)\big|(f,o)\in I\},$$ a closed subspace of $L^2_{\mathcal{O}}(D)$.
It is known that there exists a neighborhood $U_0{\subset}{\subset}D$ of $o$, integer $k_0>0$ and some constant $C_0>1$ such that
$$\sum\limits_{k=1}^{\infty}|e_k|^2\leq C_0\cdot\sum_{k=1}^{k_0}|e_k|^2\quad\mbox{on}\ U_0.$$
one can see the detail in Lemma \ref{bergman}.

Put
$$C=C_{\varepsilon_0}(\varphi):=\big[(\frac{e^{(\varepsilon_0+1)(t_0+1)}}{\varepsilon_0}C_0\sum\limits_{k=1}^{k_0}
\int_{D}\mathbbm{1}_{\{\varphi<-t_{0}\}}|e_k|^{2}d\lambda_{n})^{-\frac{1}{2}}-(1+(\frac{e^{(\varepsilon_0+1)(t_0+1)}}
{\varepsilon_0})^{-\frac{1}{2}})\big]^{-1},$$
where $t_0,\varepsilon_0$ are two positive numbers and $\varphi$ is negative on $D$ with $\varphi(0)=-\infty$.

In the present article, we obtain the following estimation of the weighted $L^2$ norm near the singularities of plurisubharmonic weight related to SOC:

\begin{thm} \label{main}
Assume that $\mathscr{I}(\varphi)_o\subset I\subset\mathcal{O}_o$.
If $C>0$, then
$$\int_{U_0\cap\{\varphi<-(t_0+1)\}}(\sum\limits_{k=1}^{\infty}|e_k|^2)e^{-\varphi}d\lambda_n<C^2.$$
\end{thm}

\begin{cor} \label{cor1}
Let $e_k\ (1\leq k\leq k_0)$ be generators of $I=\mathscr{I}(\varphi)_o$ with bounded $\sum_{k=1}^{k_0}|e_k|$ on $D$, which is in the unit ball $B(o;1)$ and $$\sum\limits_{k=1}^{k_0}\int_{D}|e_k|^{2}e^{-(1+\varepsilon_0)\varphi}d\lambda_{n}<\infty.$$
Then, for any $M>0$, there exists $t_0\gg0$ such that for any negative plurisubharmonic function $\psi$ on $D$ with $\mathscr{I}(\psi)_o\subset\mathscr{I}(\varphi)_o$ and
$$\sum\limits_{k=1}^{k_0}\int_{D}\mathbbm{1}_{\{\tilde\psi<-t_{0}\}}|e_k|^{2}d\lambda_{n}\leq2\sum\limits_{k=1}^{k_0}
\int_{D}\mathbbm{1}_{\{\tilde\varphi<-t_{0}\}}|e_k|^{2}d\lambda_{n},\qquad (*)$$
we have
$$\int_{U_0\cap\{|z|<e^{-\frac{(1+\varepsilon_0)(1+\varepsilon_0/2)(t_0+1)}{\varepsilon_0/2}}\}}(\sum\limits_{k=1}^{\infty}
|e_k|^2)e^{-\tilde\psi}d\lambda_n<M,$$
where
$$\tilde\varphi=\varphi+\frac{\varepsilon_0/2}{(1+\varepsilon_0)(1+\varepsilon_0/2)}\frac{\log|z|}{2},\
\tilde\psi=\psi+\frac{\varepsilon_0/2}{(1+\varepsilon_0)(1+\varepsilon_0/2)}\frac{\log|z|}{2}.$$
\end{cor}

By the truth of SOC and the above Corollary, we have the following convergence of the weighted $L^2$ norm related to SOC:

\begin{cor} \label{cor2}
Let $(\varphi_j)_{j\in\mathbb{N}^+}$ be a sequence of negative plurisubharmonic functions on $D$, which is convergent to $\varphi$ in Lebesgue measure, and $\mathscr{I}(\varphi_j)_o\subset\mathscr{I}(\varphi)_o$. Let $(F_j)_{j\in\mathbb{N}^+}$ be a sequence of holomorphic functions on $D$ with $(F_j,o)\in\mathscr{I}(\varphi)_o$, which is compactly convergent to a holomorphic function $F$. Then,
$|F_j|^2e^{-\varphi_j}$ converges to $|F|^2e^{-\varphi}$ in the $L^1_{loc}$ norm near $o$. In particular, there exists $\varepsilon_0>0$ such that $\mathscr{I}(\varphi_j)_o=\mathscr{I}((1+\varepsilon_0)\varphi_j)_o=\mathscr{I}(\varphi)_o$ for any large enough $j$.
\end{cor}

The last conclusion in the above Corollary can be obtained by Proposition 1.8 in \cite{G-Z_effect} and finite generation of $\mathscr{I}(\varphi)_o$.

\begin{rem}
Let $(\varphi_j)_{j\in\mathbb{N}^+}$ be a sequence of negative plurisubharmonic functions on $D$. If $\varphi_j$ is convergent to $\varphi$ in Lebesgue measure, then $\varphi_j$ converges to $\varphi$ in the $L^p_{loc}\ (0<p<\infty)$ norm.
\end{rem}

\begin{proof}
It suffices to prove $p\in\mathbb{N}^+$. By a small enough multiplication, it is enough to assume the Lelong number $\nu(\varphi,o)<1$. Thus, $\mathscr{I}(\varphi_j)_o\subset\mathscr{I}(\varphi)_o=\mathcal{O}_o$. Then, the desired result follows from Corollary \ref{cor2} and the inequality $$\frac{1}{p!}\int_D|\varphi_j-\varphi|^pd\lambda_n\leq\int_D|e^{-\varphi_j}-e^{-\varphi}|d\lambda_n.$$
which follows from the inequality $\frac{1}{p!}(a-b)^p\leq(e^{a-b}-1)e^b$, for any $a\geq b\geq0$.
\end{proof}

\section{Lemmas used in the proof of main results}

We are now in a position to prove the following Lemma.

\begin{lem} \label{bergman}
Let $I\subset\mathcal{O}_o$ be an ideal and $(e_k)_{k\in\mathbb{N}^+}$ an orthonormal basis of $$\mathcal{H}_I:=\{f{\in}L^2_{\mathcal{O}}(D)\big|(f,o)\in I\},$$ a closed subspace of $L^2_{\mathcal{O}}(D)$.
Then, there exists a neighborhood $U_0{\subset}{\subset}D$ of $o$, integer $k_0>0$ and some constant $C_0>1$ such that
$$\sum\limits_{k=1}^{\infty}|e_k|^2\leq C_0\cdot\sum_{k=1}^{k_0}|e_k|^2\quad\mbox{on}\ U_0.$$
\end{lem}

\begin{proof}
It follows from the strong Noetherian property of coherent analytic sheaves that the sequence of ideal sheaves generated by the holomorphic functions
$$(e_k(z)\overline{e_k(\overline w)})_{k\leq N},\ N=1,2,...,$$ on $D{\times}D$ is locally stationary.

Let $U{\subset}{\subset}D$ be a neighborhood of the origin $o$. Then there exists $k_0>0$ such that for any
$N\geq k_0$ we have $(e_k(z)\overline{e_k(\overline w)})_{k\leq N}=(e_k(z)\overline{e_k(\overline w)})_{k\leq k_0}$
on $U$. Since
$$|\sum\limits_{k=1}^{\infty}e_k(z)\overline{e_k(\overline w)}|\leq\big(\sum\limits_{k=1}^{\infty}|e_k(z)|^2 \sum\limits_{k=1}^{\infty}|e_k(\overline w)|^2\big)^{\frac{1}{2}},$$
then $\sum\limits_{k=1}^{\infty}e_k(z)\overline{e_k(\overline w)}$ is uniformly convergent on every compact subset of $D{\times}D$. By the closedness of coherent ideal sheaves under the topology of compact convergence (see \cite{G-R}), $\sum\limits_{k=1}^{\infty}e_k(z)\overline{e_k(\overline w)}$ is a section of the coherent ideal sheaf generated by $(e_k(z)\overline{e_k(\overline w)})_{k\leq k_0}$ over $U{\times}U$. Then, there exists a neighborhood $U_0{\subset}{\subset}U$ of $o$ and functions $a_k(z,w)\in\mathcal{O}(\overline{U_0{\times}U_0}),\ 1\leq k\leq k_0$, such that on $\overline{U_0{\times}U_0}$
$$\sum\limits_{k=1}^{\infty}e_k(z)\overline{e_k(\overline w)}=
\sum\limits_{k=1}^{k_0}a_k(z,w)e_k(z)\overline{e_k(\overline w)}.$$
Finally, by restricting to the conjugate diagonal $w=\overline z$, we get
$$\sum\limits_{k=1}^{\infty}|e_k|^2\leq C_0\cdot\sum_{k=1}^{k_0}|e_k|^2\quad\mbox{on}\ U_0.$$
\end{proof}

To prove Theorem \ref{main}, we also need the following Lemma, whose various forms already appear in \cite{G-Z, G-Z_optimal, G-Z_open_b, G-Z_effect}.

\begin{lem} \label{G-Z}
Let $B_{0}\in(0,1]$ be arbitrarily given and $t_{0}$ a positive number. Let $D_{v}$ be a strongly pseudoconvex domain relatively compact in
$\Delta^{n}$ containing $o$. Let $F$ be a holomorphic function on $\Delta^{n}$. Let $\varphi,\psi$ be two negative plurisubharmonic functions on $\Delta^{n}$, such that $\varphi(o)=\psi(o)=-\infty$. Then there exists a holomorphic function $F_{v,t_{0}}$ on $D_{v}$, such that,$$(F_{v,t_{0}}-F,o)\in\mathscr{I}(\varphi+\psi)_{o}$$ and
\begin{equation}
\label{equ:3.4}
\begin{split}
&\int_{ D_v}|F_{v,t_0}-(1-b_{t_0}(\psi))F|^{2}e^{-\varphi}d\lambda_{n}\\
\leq&(1-e^{-(t_{0}+B_{0})})\int_{D_v}\frac{1}{B_{0}}(\mathbbm{1}_{\{-t_{0}-B_{0}<\psi<-t_{0}\}})|F|^{2}
e^{-\varphi-\psi}d\lambda_{n},
\end{split}
\end{equation}
where $b_{t_{0}}(t)=\int_{-\infty}^{t}\frac{1}{B_{0}}\mathbbm{1}_{\{-t_{0}-B_{0}< s<-t_{0}\}}ds$.
\end{lem}

In particular, given $\varepsilon_0>0$ and replacing $B_0, t_0, \psi$ by $\varepsilon_0, \varepsilon_0t_0, \varepsilon_0\varphi$ respectively, we have
\begin{equation}
\begin{split}
&\int_{ D_v}|F_{v,t_0}-(1-b_{t_0}(\varepsilon_0\varphi))F|^{2}e^{-\varphi}d\lambda_{n}\\
\leq&\frac{1-e^{-(t_{0}+1)\varepsilon_0}}{\varepsilon_0}\int_{D_v}\mathbbm{1}_{\{-(t_{0}+1)<\varphi<-t_{0}\}}|F|^{2}
e^{-\varphi-\varepsilon\varphi}d\lambda_{n}\\
\leq&\frac{1}{\varepsilon_{0}}\int_{D_v}\mathbbm{1}_{\{-(t_{0}+1)<\varphi<-t_{0}\}}|F|^{2}
e^{(\varepsilon_0+1)(t_0+1)}d\lambda_{n}.
\end{split}
\end{equation}

The following Lemma is well known in real analysis.

\begin{lem} \label{weakcpt}
Let $(f_j)_{j\in\mathbb{N}^+}$ be a sequence of functions in $L^p_{loc}(D)\ (p>1)$, which is convergent to $f$ in Lebesgue measure. If there exists some constant $M>0$ such that $$\big(\int_D|f_j|^pd\lambda_n\big)^{\frac{1}{p}}<M,$$
then $$\int_D|f_j-f|d\lambda_n\to0\quad(j\to\infty).$$
\end{lem}

\section{the proof of main results}

\noindent{\textbf{\emph{Proof of Theorem} \ref{main}.} Following from Lemma \ref{G-Z}, for any $1\leq k\leq k_0$,
there exists a holomorphic function $F_k\in\mathcal{O}(D)$ such that
\begin{equation}
\begin{split}
&\int_{D}|F_k-(1-b_{t_0}(\varepsilon_0\varphi))e_k|^{2}e^{-\varphi}d\lambda_{n}\\
&\leq\frac{1}{\varepsilon_{0}}\int_{D}\mathbbm{1}_{\{-(t_{0}+1)<\varphi<-t_{0}\}}
 |e_k|^{2}e^{(\varepsilon_0+1)(t_0+1)}d\lambda_{n}.
\end{split}
\end{equation}

By Minkowski's inequality, we obtain
\begin{equation}
\begin{split}
&\big(\sum\limits_{k=1}^{k_0}\int_{D}|F_k|^2d\lambda_{n}\big)^{\frac{1}{2}}\\
\leq&\big(\sum\limits_{k=1}^{k_0}\int_{D}|F_k-(1-b_{t_0}(\varepsilon_0\varphi))e_k|^{2}e^{-\varphi}d\lambda_{n}\big)
 ^{\frac{1}{2}}+\big(\sum\limits_{k=1}^{k_0}\int_{D}|(1-b_{t_0}(\varepsilon_0\varphi))e_k|^2d\lambda_{n}\big)^{\frac{1}{2}}
\end{split}
\end{equation}
It follows from (3) and $0\leq1-b_{t_0}(\varepsilon_0\varphi)\leq\mathbbm{1}_{\{\varphi<-t_{0}\}}$ that
\begin{equation}
\begin{split}
&\big(\sum\limits_{k=1}^{k_0}\int_{D}|F_k|^2d\lambda_{n}\big)^{\frac{1}{2}}\\
\leq&\big(\frac{1}{\varepsilon_0}\sum\limits_{k=1}^{k_0}\int_{D}\mathbbm{1}_{\{-(t_{0}+1)<\varphi<-t_{0}\}}|e_k|^{2}
 e^{(\varepsilon_0+1)(t_0+1)}d\lambda_{n}\big)^{\frac{1}{2}}+\big(\sum\limits_{k=1}^{k_0}
 \int_{D}\mathbbm{1}_{\{\varphi<-t_{0}\}}|e_k|^2d\lambda_{n}\big)^{\frac{1}{2}}\\
\leq&\big((\frac{e^{(\varepsilon_0+1)(t_0+1)}}{\varepsilon_0})^{\frac{1}{2}}+1\big)\big(\sum\limits_{k=1}^{k_0}
 \int_{D}\mathbbm{1}_{\{\varphi<-t_{0}\}}|e_k|^2d\lambda_{n}\big)^{\frac{1}{2}}.
\end{split}
\end{equation}

By Lemma \ref{G-Z}, we know that
$(F_k-e_k,o)\in\mathscr{I}((1+\varepsilon_0)\varphi)_o\subset\mathscr{I}(\varphi)_o\subset I$ and $(F_k,o)\in I$.

Hence, we have
$$F_k=\sum\limits_{j=1}^{\infty}a_k^je_j, \quad a_k^j\in\mathbb{C},\ 1\le k\le k_0,$$
and $$\int_{D}|F_k|^2d\lambda_{n}=\sum\limits_{j=1}^{\infty}|a_k^j|^2,\quad 1\le k\le k_0.$$

By Lemma \ref{bergman}, the following holds on $U_0$,
\begin{equation}
\begin{split}
&\big(\sum\limits_{k=1}^{k_0}|F_k-e_k|^2\big)^{\frac{1}{2}}
\geq\big(\sum\limits_{k=1}^{k_0}|e_k|^{2}\big)^{\frac{1}{2}}-\big(\sum\limits_{k=1}^{k_0}|F_k|^2\big)^{\frac{1}{2}}\\
\geq&\big(\frac{1}{C_0}\big)^{\frac{1}{2}}\big(\sum\limits_{k=1}^{\infty}|e_k|^{2})\big)^{\frac{1}{2}}-\big(\sum\limits_{k=1}^{k_0}
(\sum\limits_{j=1}^{\infty}|a_k^j|^2)\big)^{\frac{1}{2}}\big(\sum\limits_{k=1}^{\infty}|e_k|^2\big)^{\frac{1}{2}}\\
=&\big((\frac{1}{C_0})^{\frac{1}{2}}-(\sum\limits_{k=1}^{k_0}\int_{D}|F_k|^2d\lambda_{n})^{\frac{1}{2}}\big)
 \big(\sum\limits_{k=1}^{\infty}|e_k|^2\big)^{\frac{1}{2}}\\
\geq&\big((\frac{1}{C_0})^{\frac{1}{2}}-((\frac{e^{(\varepsilon_0+1)(t_0+1)}}{\varepsilon_0})^{\frac{1}{2}}+1)
 (\sum\limits_{k=1}^{k_0}\int_{D}\mathbbm{1}_{\{\varphi<-t_{0}\}}|e_k|^2d\lambda_{n})^{\frac{1}{2}}\big)
 \big(\sum\limits_{k=1}^{\infty}|e_k|^2\big)^{\frac{1}{2}}.
\end{split}
\end{equation}

Denote by
$$A:=(\frac{1}{C_0})^{\frac{1}{2}}-((\frac{e^{(\varepsilon_0+1)(t_0+1)}}{\varepsilon_0})^{\frac{1}{2}}+1)
(\sum\limits_{k=1}^{k_0}\int_{D}\mathbbm{1}_{\{\varphi<-t_{0}\}}|e_k|^2d\lambda_{n})^{\frac{1}{2}}.$$
Since $C_{\varepsilon_0}(\varphi)>0$ and
$$A{\cdot}C_{\varepsilon_0}(\varphi)=\big(\frac{e^{(\varepsilon_0+1)(t_0+1)}}{\varepsilon_{0}}\sum\limits_{k=1}^{k_0}
\int_{D}\mathbbm{1}_{\{\varphi<-t_{0}\}}|e_k|^2d\lambda_{n}\big)^{\frac{1}{2}}>0,$$
it follows that $A>0$.

Then from (6) we obtain
\begin{equation}
\begin{split}
&A^2\cdot\big(\int_{\{\varphi<-(t_0+1)\}{\cap}U_0}(\sum\limits_{k=1}^{\infty}|e_k|^2)e^{-\varphi}d\lambda_{n}\big)\\
\leq&\int_{\{\varphi<-(t_0+1)\}{\cap}U_0}(\sum\limits_{k=1}^{k_0}|F_k-e_k|^2)e^{-\varphi}d\lambda_{n}\\
=&\sum\limits_{k=1}^{k_0}\int_{\{\varphi<-(t_0+1)\}{\cap}U_0}|F_k-e_k|^2e^{-\varphi}d\lambda_{n}.
\end{split}
\end{equation}

Note that
$$\sum\limits_{k=1}^{k_0}|F_k-(1-b_{t_0}(\varepsilon_0\varphi))e_k|^{2}\big|_{\{\varphi<-(t_0+1)\}{\cap}U_0}
=\sum\limits_{k=1}^{k_0}|F_k-e_k|^{2}.$$
It follows from Lemma \ref{G-Z} that
\begin{equation}
\begin{split}
&\sum\limits_{k=1}^{k_0}\int_{\{\varphi<-(t_0+1)\}{\cap}U_0}|F_k-e_k|^2e^{-\varphi}d\lambda_{n}\\
\leq&\sum\limits_{k=1}^{k_0}\int_{D}|F_k-(1-b_{t_0}(\varepsilon_0\varphi))e_k|^{2}e^{-\varphi}d\lambda_{n}\\
\leq&\frac{1}{\varepsilon_{0}}\sum\limits_{k=1}^{k_0}\int_{D}\mathbbm{1}_{\{-(t_{0}+1)<\varphi<-t_{0}\}}|e_k|^{2}
 e^{(\varepsilon_0+1)(t_0+1)}d\lambda_{n}\\
\leq&\frac{e^{(\varepsilon_0+1)(t_0+1)}}{\varepsilon_{0}}\sum\limits_{k=1}^{k_0}\int_{D}\mathbbm{1}_
 {\{\varphi<-t_{0}\}}|e_k|^{2}d\lambda_{n}.
\end{split}
\end{equation}

Combining inequalities (7) and (8), we have
\begin{equation}
\begin{split}
&\int_{\{\varphi<-(t_0+1)\}{\cap}U_0}(\sum\limits_{k=1}^{\infty}|e_k|^2)e^{-\varphi}d\lambda_{n}\\
\leq&\frac{1}{A^2}\cdot\frac{e^{(\varepsilon_0+1)(t_0+1)}}{\varepsilon_{0}}\sum\limits_{k=1}^{k_0}\int_{D}\mathbbm{1}_
 {\{\varphi<-t_{0}\}}|e_k|^{2}d\lambda_{n}\\
=&\big[(\frac{e^{(\varepsilon_0+1)(t_0+1)}}{\varepsilon_0}C_0\sum\limits_{k=1}^{k_0}\int_{D}\mathbbm{1}_
 {\{\varphi<-t_{0}\}}|e_k|^{2}d\lambda_{n})^{-\frac{1}{2}}-(1+(\frac{e^{(\varepsilon_0+1)(t_0+1)}}{\varepsilon_0})^
 {-\frac{1}{2}})\big]^{-2}\\
=&C^2_{\varepsilon_0}(\varphi).
\end{split}
\end{equation}

\hfill $\Box$

\begin{rem} \label{remark}
If $I=\mathscr{I}(\varphi)_o=\mathscr{I}((1+\varepsilon_0)\varphi)_o$ in the above theorem and $$\sum\limits_{k=1}^{k_0}\int_D|e_k|^2e^{-(1+\varepsilon_0)\varphi}d\lambda_{n}<\infty,$$
then for any $\varepsilon_1,\ \varepsilon_2>0$, there exists $t_0\gg0$ such that
$$\sum\limits_{k=1}^{k_0}\int_D\mathbbm{1}_{\{-(t_0+1)<\varphi<-t_0\}}|e_k|^2e^{(\varepsilon_0+1)(t_0+1)}d\lambda_{n}<\varepsilon_1$$
and $$\sum\limits_{k=1}^{k_0}\int_{D}|(1-b_{t_0}(\varepsilon_0\varphi))e_k|^2d\lambda_{n}<\varepsilon_2.$$
Furthermore, we can get
$$\int_{\{\varphi<-(t_0+1)\}{\cap}U_0}(\sum\limits_{k=1}^{\infty}|e_k|^2)e^{-\varphi}d\lambda_{n}\leq\big((\frac{\varepsilon_1}
{\varepsilon_0}\cdot C_0)^{-1/2}-(1+(\frac{\varepsilon_1}{\varepsilon_0\varepsilon_2})^{-1/2})\big)^{-2}.$$
\end{rem}

\medskip

\noindent{\textbf{\emph{Proof of Corollary} \ref{cor1}.} By H\"older inequality, we have
\begin{equation*}
\begin{split}
&\int_U|F|^2e^{-(1+\varepsilon_0/2)\tilde\varphi}d\lambda_n\\
&\leq\big(\int_U|F|^2e^{-(1+\varepsilon_0)\varphi}d\lambda_n\big)^{\frac{1+\varepsilon_0/2}{1+\varepsilon_0}}
\big(\int_U|F|^2e^{-\log|z|/2}d\lambda_n\big)^{\frac{\varepsilon_0/2}{1+\varepsilon_0}},
\end{split}
\end{equation*}
which implies $\mathscr{I}((1+\varepsilon_0)\varphi)_o\subset\mathscr{I}((1+\varepsilon_0/2)\tilde\varphi)_o\subset
\mathscr{I}(\tilde\varphi)_o\subset\mathscr{I}(\varphi)_o$,
i.e., $$\mathscr{I}((1+\varepsilon_0/2)\tilde\varphi)_o=\mathscr{I}(\tilde\varphi)_o=\mathscr{I}(\varphi)_o.$$
As
$$\sum\limits_{k=1}^{k_0}\int_{D}|e_k|^{2}e^{-(1+\varepsilon_0)\varphi}d\lambda_{n}<\infty,$$
there exists $t_0{\gg}0$ such that $0{<}C_{\varepsilon_0/2}(\tilde\varphi){<}\sqrt{M}/2$, and $0<C_{\varepsilon_0/2}(\tilde\psi)\leq2\cdot C_{\varepsilon_0/2}(\tilde\varphi)$ by $(*)$.

Since $\tilde\psi\leq\frac{\varepsilon_0/2}{(1+\varepsilon_0)(1+\varepsilon_0/2)}\log|z|$ on $D$, we have
$$\{|z|<e^{-\frac{(1+\varepsilon_0)(1+\varepsilon_0/2)(t_0+1)}{\varepsilon_0/2}}\}\subset\{\tilde\psi<-(t_0+1)\}.$$
Then, we obtain that
\begin{equation*}
\begin{split}
&\int_{U_0\cap\{|z|<e^{-\frac{(1+\varepsilon_0)(1+\varepsilon_0/2)(t_0+1)}{\varepsilon_0/2}}\}}(\sum\limits_{k=1}^{\infty}
|e_k|^2)e^{-\tilde\psi}d\lambda_n\\
\leq&\int_{U_0\cap\{\tilde\psi<-(t_0+1)\}}(\sum\limits_{k=1}^{\infty}|e_k|^2)e^{-\tilde\psi}d\lambda_n
 \leq C_{\varepsilon_0/2}^2(\tilde\psi)<M.
\end{split}
\end{equation*}

\hfill $\Box$
\\
\noindent{\textbf{\emph{Proof of Corollary} \ref{cor2}.} As every sequence which is convergent in Lebesgue measure has a subsequence which is convergent almost everywhere, then it is sufficient to prove the result for the case that $\varphi_j$ is convergent to $\varphi$ almost everywhere.

By the truth of SOC, there exists $\varepsilon_0>0$ such that $\mathscr{I}(\varphi)=\mathscr{I}((1+\varepsilon_0)\varphi)$ on a neighborhood $D$ of $o$. Without loss of generality, we assume the unit ball $B(o;1)\supset D$.

Since $F_j$ is compactly convergent to a holomorphic function $F$, by shrinking $D$, we can assume that $\int_D|F_j|^2d\lambda_n$ is uniformly bounded. Let $e_k,\ 1\leq k\leq k_0,$ be as in Corollary \ref{cor1}. Then, we infer from $(F_j,o)\in\mathscr{I}(\varphi)_o$ and Lemma \ref{bergman} that there exist complex numbers $a_j^k$ such that $F_j=\sum\limits_{k=1}^{\infty}a_j^ke_k$ and $\sum\limits_{k=1}^{\infty}|a_{j}^{k}|^2=\int_D|F_j|^2d\lambda_n$ is uniformly bounded.

Since $\varphi_j$ is convergent to $\varphi$ almost everywhere, it follows from the dominated convergence theorem that
$$\sum\limits_{k=1}^{k_0}\int_{D}\mathbbm{1}_{\{\tilde{\varphi}_j<-t_{0}\}}|e_k|^{2}d\lambda_{n}\leq2\sum\limits_{k=1}^{k_0}
\int_{D}\mathbbm{1}_{\{\tilde\varphi<-t_{0}\}}|e_k|^{2}d\lambda_{n},$$
where $\tilde{\varphi_j}=\varphi_j+\frac{\varepsilon_0/2}{(1+\varepsilon_0)(1+\varepsilon_0/2)}\frac{\log|z|}{2}$.
By Corollary \ref{cor1}, there exists a neighborhood $V_0{\subset}{\subset}D$ of $o$ and $M>0$ such that $$\int_{V_0}\sum\limits_{k=1}^{\infty}|e_k|^{2}e^{-\varphi_j}d\lambda_{n}<M.$$

Let $\varepsilon\in(0,\varepsilon_0)$. Replacing $\varphi$ by $(1+\varepsilon/2)\varphi$ and $\varphi_j$ by $(1+\varepsilon/2)\varphi_j$, we have
$$\int_{\tilde V_0}\sum\limits_{k=1}^{\infty}|e_k|^{2}e^{-(1+\varepsilon/2)\varphi_j}d\lambda_{n}<\tilde M,$$
for some neighborhood $\tilde V_0\supset o$ and some constant $\tilde M$ which are independent of $\varphi_j$.
As $\sum\limits_{k=1}^{\infty}|a_{j}^{k}|^2$ is uniformly bounded, by Schwarz inequality, it follows that
$$\int_{\tilde V_0}|F_j|^2e^{-(1+\varepsilon/2)\varphi_j}d\lambda_{n}\leq\int_{\tilde V_0}(\sum\limits_{k=1}^{\infty}|a_{j}^{k}|^2)\cdot(\sum\limits_{k=1}^{\infty}|e_{k}|^2)e^{-(1+\varepsilon/2)\varphi_j}d\lambda_{n}$$
is uniformly bounded. Then, by Lemma \ref{weakcpt}, we obtain that
$F_je^{-\varphi_j}$ converges to $Fe^{-\varphi}$ as $j$ goes to infinity in the $L^1_{loc}$ norm on $\tilde V_0$.

Replacing $\varphi_j$ by $(1+\varepsilon_0)\varphi_j$, we obtain the second assertion from the first one.\hfill $\Box$

\section{Relation to semi-continuity of complex singularity exponents}

In \cite{D-K}, Demailly and Koll\'{a}r proved the following semi-continuity of complex singularity exponents.

\begin{thm} \emph{(Main Theorem 0.2, \cite{D-K}).} \label{D-K}
Let $X$ be a complex manifold, $K\subset X$ a compact subset and $\varphi$ a plurisubharmonic function on $X$. If $c<c_K(\varphi)$ and $(\varphi_j)$ is a sequence of plurisubharmonic functions on $X$ which is convergent to $\varphi$ in $L_{loc}^1$ norm, then $e^{-2c\varphi_j}$ converges to $e^{-2c\varphi}$ in $L^1$ norm over some neighborhood $U$ of $K$.
\end{thm}

Indeed, by subtracting a constant, we can assume $\varphi$ is negative on $K$. As $\int_K\varphi_jd\lambda_n\leq\int_K|\varphi-\varphi_j|d\lambda_n+\int_K\varphi d\lambda_n$, we obtain that $\varphi_j$ is also negative on $K$. Then, Theorem \ref{D-K} is a special case of Corollary \ref{cor2} when $\mathscr{I}(\varphi)_o=\mathcal{O}_o$. With additional condition $\varphi_j\leq\varphi$, Theorem \ref{D-K} can be referred to \cite{Pham} for multiplier ideals, which is also a special case of Corollary \ref{cor2}.

If $\varphi=\log|g|$ for $J$-vector holomorphic functions $g(z,c)=(g_1,...,g_J)$ on a polydisk $\Delta^{n}\times\Delta$, then it follows that

\begin{thm} \emph{(Main Theorem, \cite{P-S}).}
Assume that $\int_{\Delta^{n}}|g(z,0)|^{-\delta}<\infty.$
Then there exists a smaller polydisk $\Delta'^{n}\times\Delta'$ so that the function
$c\mapsto\int_{\Delta^{'n}}|g(z,c)|^{-\delta}$ is finite and continuous for $c\in\Delta'$.
\end{thm}


\end{document}